\documentclass{amsart}

\usepackage{amsfonts,amssymb,mathrsfs}
\usepackage{graphics,color,hyperref}
\providecommand{\arxiv}[2][]{\href{http://www.arXiv.org/abs/#2}{arXiv:#2}}
\newtheorem{theorem}{Theorem}

\newtheorem{corollary}{Corollary}
\newtheorem{proposition}{Proposition}  
\newtheorem{lemma}{Lemma} 
\theoremstyle{remark}
\newtheorem{remark}{Remark} 
\newtheorem{example}{Example}
\title[The tropical analogue of polar cones]{The tropical analogue of polar cones}
\author{St\'ephane Gaubert}
\thanks{The first author was partially 
supported by the joint CNRS-RFBR grant number 05-01-02807}
\address{INRIA and Centre de Math\'ematiques Appliqu\'ees, 
\'Ecole Polytechnique. Postal address: CMAP, \'Ecole Polytechnique, 
91128 Palaiseau C\'edex, France. T\`el: +33 1 69 33 46 13, 
Fax: +33 1 39 63 57 86}
\email{Stephane.Gaubert@inria.fr }
\author{Ricardo D. Katz}
\address{CONICET. Postal address: Instituto de Matem\'atica ``Beppo Levi'', 
Universidad Nacional de Rosario, Avenida Pellegrini 250, 2000 Rosario, 
Argentina.}
\email{rkatz@fceia.unr.edu.ar}
\DeclareMathAlphabet{\mathbbold}{U}{bbold}{m}{n}
\newcommand{\eme}[1]{\hat{#1}}

\newcommand{\emb}[1]{\bar{#1}}
\newcommand{\wemb}[1]{\overline{#1}}
\newcommand{\zero}{\varepsilon}
\newcommand{\unit}{e}

\newcommand{\set}[2]{\{#1\mid\,#2\}}
\newcommand{\bracket}[2]{\langle#1\mid #2\rangle}

\newcommand{\cpol}[1]{\overline{\rm pol}(#1)}
\newcommand{\ccong}[1]{\overline{\rm cong}(#1)}
\newcommand{\opp}[1]{#1^{\rm op}}

\newcommand{\R}{\mathbb{R}}
\newcommand{\N}{\mathbb{N}}
\newcommand{\B}{\mathbb{B}}
\newcommand{\Rm}{\R\cup\{-\infty\}}
\newcommand{\rmax}{\R_{\max}}
\newcommand{\rmaxb}{\overline{\R}_{\max}}
\newcommand{\cX}{X}

\newcommand{\uvector}{{\rm e}}
\newcommand{\cK}{\mathcal{K}}

\begin{document}

\begin{abstract}
We study the max-plus or tropical analogue of the notion of polar:
the polar of a cone represents the set of linear inequalities satisfied 
by its elements. We establish an analogue of the bipolar theorem, which
characterizes all the inequalities satisfied by the elements
of a tropical convex cone. We derive this characterization from a new
separation theorem. We also establish variants of these results
concerning systems of linear equalities.
\end{abstract}

\maketitle

\section{Introduction}

Max-plus or tropical algebra refers to the analogue of classical
algebra obtained by considering the max-plus semiring $\rmax$,
which is the set $\Rm$ equipped with the addition 
$a\oplus b:=\max (a,b)$ 
and the multiplication $a\otimes b:=a+b$.

Max-plus {\em convex cones} or {\em semimodules} are 
subsets $V$ of $\rmax^n$ stable by max-plus linear 
combinations, meaning that 
\begin{align}\label{closure}
\max(\lambda+x,\mu+y)\in V
\end{align} 
for all $x,y\in V$ and $\lambda,\mu \in \rmax$. Here, $\lambda +x$
denotes the vector with entries $\lambda+x_i$ and the ``max''
of two vectors is understood entrywise.
Max-plus {\em convex sets} are subsets $V$ of $\rmax^n$ which 
satisfy~\eqref{closure} for all  $x,y\in V$ and $\lambda,\mu \in \rmax$ 
such that $\max (\lambda , \mu)=0$. 

Max-plus convex sets and cones have been studied
by several authors with different motivations. 

They were introduced by K. Zimmermann~\cite{zimmerman77}
when studying discrete optimization problems.

Another motivation arose
from the use by Maslov~\cite{maslov73}
of a max-plus analogue of the superposition
principle, which applies to the
solutions of a stationary Hamilton-Jacobi equation
whose Hamiltonian is convex in the adjoint variable.
Hence, the space of solutions has a structure analogous
to that of a linear space.
This was one of the motivations for the development
of the Idempotent Analysis~\cite{maslov92,maslovkolokoltsov95}
which includes the study, by Litvinov, Maslov, and Shpiz~\cite{litvinov00}
of the idempotent spaces,
of which max-plus cones are special cases. See~\cite{LS}
for a recent development.

A third motivation comes from the
algebraic approach to discrete event systems initiated
by Cohen, Dubois, Quadrat and Viot~\cite{cohen85a}. 
The spaces that appear in the study
of max-plus linear dynamical systems (reachable and observable spaces)
are naturally equipped with structures of 
max-plus cones~\cite{ccggq99}. This motivated the study of
max-plus cones or semimodules 
by Cohen, Quadrat, and the first author~\cite{cgq00,cgq02}. 
The results of~\cite{cgq02} were completed in a work with Singer~\cite{cgqs04}.
A development by the second author~\cite{katz05}
of the work on discrete event systems
includes a max-plus cone based synthesis
of controllers. See also~\cite{gk02a}.

The theory of max-plus convex cones, or tropical convex sets, has
recently been developed
in relation to tropical geometry. The tropical analogues of polytopes 
were considered by Develin and Sturmfels~\cite{sturmfels03}, 
who established a remarkable connection with phylogenetic analysis,
showing that tree metrics may be thought of in terms of tropical polytopes.
Tropical convexity was further studied by
Joswig~\cite{joswig04} and Develin and Yu~\cite{DevelinYu}. 
See~\cite{JSY} for a new development.

Another interest in the max-plus analogues of convex cones comes from 
abstract convex analysis~\cite{ACA}, as can be seen in the work of 
Nitica and Singer~\cite{NiticaSinger07a,NiticaSinger07b,NiticaSinger07c}.
Independently, 
Briec, Horvath and Rubinov~\cite{BriecHorvath04,BriecHorvRub05}
introduced and studied the notion of 
$\B$-convexity which is another name for max-plus convexity. 
See also~\cite{AdilovRubinov06}.

These motivations led to the proof of max-plus analogues
of classical results such as: Hahn-Banach theorem either in analytic 
(extension of linear forms) or geometric (separation) 
form~\cite{zimmerman77,shpiz,cgq00,sturmfels03,cgq02,BriecHorvRub05,cgqs04}, 
Minkowski theorem (generation of a closed convex set by its 
extreme points)~\cite{helbig,GK06a,BSS-07,minkowski}, and 
Helly and Carath\'eodory type 
theorems~\cite{BriecHorvath04,GaubertSergeev,Meunier}. 
However, some duality results are still
missing, because the idempotency of addition creates
difficulties which are absent in the classical theory. 
For example, in the classical theory there exists a 
bijective correspondence between the facets of a 
convex cone and the extreme rays of its polar cone. 
This kind of properties, relating internal and external 
representations of convex cones, are not yet well understood 
in the max-plus case. 

Indeed, classical duality theory relates a convex cone
with the set of linear inequalities that its points
satisfy. In the max-plus setting,  
the {\em polar} of a max-plus convex cone $V\subset \rmax^n$ 
can be defined as 
\[
V^\circ =\set{(f,g)\in (\rmax^n)^2}{\max_{i} \left( f_i + x_i \right) 
\leq \max_j \left( g_j + x_j \right) ,\; \forall x\in V} \; .
\]
Conversely, we may consider the set of points satisfying
a given set of inequalities. This leads us to define,
for all $W\subset (\rmax^n)^2$, a ``dual'' polar cone:
\[
W^\diamond = \set{x\in \rmax^n}{\max_{i} \left( f_i + x_i \right) 
\leq \max_j \left( g_j + x_j \right) ,\; \forall (f,g)\in W} \; .
\]
It follows from the separation theorem of~\cite{zimmerman77,shpiz,cgqs04} 
that a closed max-plus convex cone $V$ is characterized
by its polar cone:
\[
(V^\circ)^\diamond = V \; .
\]
It is natural to ask which subsets $W$ of $(\rmax^n)^2$
arise as polars of cones, or equivalently,
which subsets are of the form $W=(U^\diamond)^\circ$ 
for some $U\subset (\rmax^n)^2$.
The main result of this paper
is the following analogue of the bipolar theorem:
\begin{theorem}[Bipolar]\label{bipolar}
A subset $W\subset (\rmax^n)^2$ is the polar of a cone if, and only if, the
following conditions are satisfied:
\begin{itemize}
\item[(i)] $W$ is a closed max-plus convex cone,
\item[(ii)]  $f\leq g \implies (f,g)\in W$,
\item[(iii)] $(f,g)\in W\; ,\; (g,h)\in W \implies (f,h)\in W$. 
\end{itemize}
\end{theorem}
This follows from Theorem~\ref{ThBipolarRmax} below.
We deduce this result from a new separation
theorem, Theorem~\ref{TSepPolRmax}, 
which is in some sense a dual of the separation
theorem of~\cite{cgqs04}.

We shall first establish similar results
when the scalars form a complete semiring
which is reflexive in the sense of~\cite{cgq02}. 
This applies in particular to the completion of the max-plus semiring $\rmaxb$.
Thanks to completeness and reflexivity properties, 
we shall obtain the separation theorem by algebraic and order
theoretical arguments, relying on residuation theory
(Galois connection in lattices). Then, we shall derive the results in the case 
of the max-plus semiring by a perturbation argument. 

The notion of polar is illustrated in Figures~\ref{polar}
 and~\ref{figuresemi} below. Figure~\ref{polar} represents a max-plus convex 
cone $V$ and eight max-plus linear inequalities composing a set $W$ 
such that $V=W^\diamond$. For visibility reasons, 
we have depicted completely only one of these inequalities, 
which is also shown separately. Detailed explanations of the 
constructions of Figure~\ref{polar} can be found in 
Example~\ref{ExamplePolar}.  

We shall show that if $W$ is a finite
subset of $(\rmax^n)^2$, representing a finite set
of linear inequalities, then there are in general elements
in the transitive closure of $W$ which are not
linear combinations of $W$. This means that there are 
some inequalities which can be logically deduced from
the finite set of inequalities 
\[
\max_i\left( f_i + x_i \right) \leq \max_j\left( g_j + x_j \right)\;  ,
\;\; (f,g)\in W 
\]
but which cannot be obtained by a linear combination of these
inequalities. Hence, Theorem~\ref{bipolar} shows that
there is no direct analogue of Farkas lemma. 

This is related to the difficulties in defining
the faces of tropical polyhedra, due to the obstructions
mentioned in~\cite{minkowski} and~\cite{DevelinYu}.

Some properties concerning max-plus convex cones 
are analogous to the case of cones in classical convexity.
However, max-plus convex cones also have some features
which are reminiscent of classical linear spaces.
This is because in max-plus algebra, an inequality
$\max_i\left( f_i + x_i \right)\leq 
\max_j \left( g_j + x_j \right) $ 
can always be rewritten as an equality,
$\max_i\left(\max (f_i , g_i ) + x_i \right) = 
\max_j \left( g_j + x_j \right) $. 
This point of view leads to define~\cite{maxplus97} 
the orthogonal of a max-plus convex cone $V$,
\[
V^\bot = \set{(f,g)\in (\rmax^n)^2}{\max_i \left( f_i + x_i \right) 
= \max_j \left( g_j + x_j \right) ,\; \forall x\in V} \; .
\]
Such an orthogonal is a congruence of semimodules, 
in the sense that will be defined in Section~\ref{SecSepTheo}.
Like in the case of linear inequalities, 
we may also consider the set of points satisfying a given set of linear 
equalities: for $W\subset (\rmax^n)^2$ define
\[
W^\top = \set{x\in \rmax^n}{\max_i \left( f_i + x_i \right) 
= \max_j \left( g_j + x_j \right) ,\; \forall (f,g)\in W} \; .
\]
As a consequence of the separation theorem 
of~\cite{zimmerman77,shpiz,cgqs04} we have  
\[ 
(V^\bot)^\top=V \; , 
\]
if $V$ is a closed cone, 
a result which was already stated in~\cite{maxplus97} 
in the particular case of finitely generated cones. 
We shall prove, as a variant of our main theorem, the following 
biorthogonal theorem. Like the bipolar theorem above 
(Theorem~\ref{bipolar}), it also follows from Theorem~\ref{ThBipolarRmax}.
\begin{theorem}[Biorthogonal]
A subset $W\subset (\rmax^n)^2$ is the orthogonal of a cone if, 
and only if, the following conditions are satisfied:
\begin{itemize}
\item[(i)] $W$ is a closed max-plus convex cone,
\item[(ii)] $(f,f)\in W \; , \; \forall f\in \rmax^n$ ,
\item[(iii)] $(f,g)\in W \implies (g,f)\in W$, 
\item[(iv)] $(f,g)\in W\; ,\; (g,h)\in W \implies (f,h)\in W$. 
\end{itemize}
\end{theorem}

Congruences
are used in~\cite{conditioned} 
to solve some observability problems
for max-plus linear discrete event systems, i.e.\ to reconstruct
some information on the state of the system from 
observations. The present biorthogonality
theorem is used there 
to construct dynamic observers:
in fact, cones are simpler to manipulate than congruences,
due to the absence of ``doubling'' of the dimension, duality
allows us  to reduce algorithmic problems
concerning congruences to algorithmic problems
concerning cones. 
The duality results
of the present paper may also be useful in the recent application of max-plus
polyhedra to static analysis by abstract interpretation~\cite{AGG08}.

\section{Preliminaries}\label{SectPrel}

In this section we recall some notions and results concerning semimodules over 
idempotent semirings which we shall need throughout this paper. We refer 
the reader to~\cite{bcoq,cgq02,BlythJan72} for more background.

A monotone map $F:T\rightarrow S$ between two ordered sets 
$(T,\leq)$ and $(S,\leq)$ is said to be {\em residuated} 
if for each $s\in S$ the least upper bound 
of the set $\set{t\in T}{F(t)\leq s}$ exists and belongs to it. 
When $F$ is residuated, the map $F^\sharp :S\rightarrow T$ defined by 
$F^\sharp (s):=\vee \set{t\in T}{F(t)\leq s}$, where $\vee Q$  
denotes the least upper bound of $Q$, is called the {\em residual} of $F$. 
We say that $F$ is {\em continuous} if $F(\vee Q)=\vee F(Q)$
for all $Q\subset T$, where $F(Q):=\set{F(t)}{t\in Q}$. 
The term continuous refers to the Scott topology~\cite{continuous},
which is a (non-Hausdorff) topology well adapted to the study of completed 
ordered algebraic structures. 
When $T$ and $S$ are complete ordered sets, meaning that any of their 
subsets admits a least upper bound, 
it can be shown that $F:T\rightarrow S$ is residuated if, 
and only if, it is continuous (see~\cite[Th.~5.2]{BlythJan72} 
or~\cite[Th.~4.50]{bcoq}). 

Recall that any idempotent semiring $({\cK},\oplus,\otimes,\varepsilon,e)$, 
where $\varepsilon$ and $e$ represent the neutral elements for 
addition $\oplus $ and multiplication $\otimes$ respectively, 
can be equipped with the {\em natural} order relation: 
$a\leq b\iff a\oplus b=b$, for which $\vee\left\{a,b\right\}=a\oplus b$. 
We say that an idempotent semiring $\cK $ is {\em complete} 
if it is complete as a naturally ordered set, 
and if the left and right multiplications, $L_a:\cK\rightarrow \cK$, 
$L_a(b)=a\otimes b$ and $R_a:\cK\rightarrow \cK$, 
$R_a(b)=b\otimes a$, are continuous.    

Given a semiring $({\cK},\oplus,\otimes,\varepsilon_{{\cK}},e)$, 
a right {\em $\cK$-semimodule} 
is a commutative monoid $({\cX},\hat{\oplus })$, 
with neutral element $\varepsilon_{{\cX}}$, equipped with a map 
${\cX}\times {\cK}\to {\cX}$, $(x,\lambda) \to x \cdot \lambda $, 
(right action), which satisfies: 
\begin{align*}
& & (x\; \hat{\oplus }\;z)\cdot  \lambda =x \cdot \lambda \; \hat{\oplus }\; 
z \cdot \lambda \; ,   
& & x\cdot (\lambda \oplus \mu)=x\cdot \lambda \; \hat{ \oplus } \; 
x \cdot \mu \; ,  
& & x \cdot (\lambda \otimes \mu)= (x \cdot \lambda )\cdot \mu \; , \\
& & x \cdot \varepsilon_{\cK}= \varepsilon_{\cX}\; ,
& & \varepsilon_{\cX} \cdot \lambda = \varepsilon_{\cX} \; ,
& & x\cdot e =x \; ,
\end{align*}
for all $x,z\in {\cX }$ and $\lambda,\mu\in {\cK}$. Henceforth, 
the semimodule addition will be denoted by $\oplus$ (instead of 
$\hat{\oplus }$), like the semiring addition, 
and we will use concatenation to denote both the multiplication 
of ${\cK}$ and the right action. Also for simplicity, 
we will denote by $\varepsilon$ both the zero 
element $\varepsilon_{{\cK}}$ of ${\cK}$ and the zero 
element $\varepsilon_{{\cX}}$ of ${\cX}$.  
Left $\cK$-semimodules are defined dually.
A {\em subsemimodule} of ${\cX}$ is a subset 
$V\subset X$ such that $x \lambda \oplus z \mu \in V$ 
for all $x,z\in V$ and $\lambda,\mu\in {\cK}$.  
Note that a right (or left) $\cK$-semimodule $(X,\oplus)$ is necessarily 
idempotent when $\cK$ is idempotent. In this case, 
$X$ is said to be {\em complete}, 
if it is complete as a naturally ordered set, 
and if for all $u\in X$ and $\lambda \in \cK$, 
the maps $R_\lambda:X\rightarrow X$, $z\mapsto z\lambda $ and 
$L_u:\cK\rightarrow X$, $\mu\mapsto u\mu $, 
are both continuous. Then, we can define 
\begin{align*}
u\backslash x:=L_u^\sharp (x)=\vee \set{\mu \in \cK}{u\mu \leq x} \; 
\makebox{ and } \;
x/\lambda:=R_\lambda ^\sharp(x)=\vee \set{z \in X}{z\lambda \leq x} \; ,
\end{align*}  
for all $x,u\in X$ and $\lambda \in \cK$. 
Note that by definition, we have $u (u\backslash x)\leq x$ and 
$(x/\lambda) \lambda \leq x$.  
A subsemimodule $V$ of $X$ is complete, 
if $\vee Q\in V$ for all $Q\subset V$. 
It is convenient to recall the following universal separation 
theorem for complete semimodules. 
\begin{theorem}[{\cite[Th.~8]{cgq02}}]\label{UnSepTheo}
Let $V\subset X$ be a complete subsemimodule. Assume that $x\not \in V$. 
Then, 
\[
v\in V \implies v\backslash P_V(x) =v\backslash x 
\] 
and 
\[
x\backslash P_V(x) \neq x\backslash x \; , 
\] 
where $P_V (x):= \vee \set{v\in V}{v\leq x}$. 
\end{theorem}

\begin{example}
The set $\rmax^n$ is a 
right $\rmax$-semimodule when it is equipped 
with the usual operations: 
$(x\oplus z)_i:=x_i \oplus z_i$ and 
$(x\lambda )_i:= x_i \lambda$ 
for all $x,z\in \rmax^n$ and $\lambda \in \rmax$. 
Since $\rmax $ is not a complete semiring, 
it will be convenient to consider the 
{\em completed} max-plus semiring $\rmaxb$, 
which is obtained by adjoining the element $+\infty$ to $\rmax$.
\end{example}

In order to state separation theorems with the usual scalar product 
form, we next recall the notion of dual pair of~\cite{cgq02}, 
which is analogous to the one that arises in the theory of 
topological vector spaces (see for example~\cite{bourbaki}). 
As usual, a map between right (resp. left) $\cK$-semimodules is 
right (resp. left) {\em linear} if it preserves finite sums and 
commutes with the action of scalars. 
We call {\em pre-dual pair} for a complete idempotent semiring $\cK$, 
a complete right $\cK$-semimodule $X$ together with a complete 
left $\cK$-semimodule $Y$ equipped with a bracket 
$\bracket{\cdot}{\cdot}$ from $Y\times X$ to $\cK$, 
such that, for all $x\in X$ and $y\in Y$, the maps 
$y\mapsto \bracket{y}{x}$ and $x\mapsto \bracket{y}{x}$ are respectively 
left and right linear, and continuous. 
We shall denote by $(Y,X)$ this pre-dual pair. 
We say that $Y$ {\em separates} $X$ if 
\[
(\forall y\in Y, \bracket{y}{x}=\bracket{y}{x'})\implies x=x' \; ,
\]
and that $X$ {\em separates} $Y$ if 
\[
(\forall x\in X, \bracket{y}{x}=\bracket{y'}{x})\implies y=y' \; .
\]
A pre-dual pair $(Y,X)$ such that $X$ separates $Y$ and $Y$ separates $X$ 
is called {\em dual pair}. 

When $X$ is a complete right $\cK$-semimodule, we call 
{\em opposite} semimodule of $X$ the left $\cK$-semimodule 
$\opp{X}$ whose set of elements is $X$ equipped with the addition 
$(x,u) \mapsto \wedge \left\{ x,u\right\} $ and the left action 
$(\lambda , x)\mapsto x/\lambda $. Here, $\wedge \left\{ x,u \right\}$ 
is the greatest lower bound of $\left\{ x,u\right\} $ for the 
natural order of $X$. The fact that $\opp{X}$ 
is a complete left $\cK$-semimodule follows from 
properties of the residual map $x/\lambda $ 
(see~\cite{cgq02} for details). In particular, 
the complete idempotent semiring $\cK$
can be thought of as a right $\cK$-semimodule, and so
the same construction and notation apply to $\cK$. 

\begin{theorem}[{\cite[Th.~22]{cgq02}}]\label{OppTheo}
Let $X$ be a complete right $\cK$-semimodule. Then, 
the semimodules $\opp{X}$ and $X$ form a dual pair for the bracket 
$\opp{X}\times X\rightarrow \opp{\cK}$, $(y,x)\mapsto \bracket{y}{x}= x\backslash y$.
\end{theorem}

Given a pre-dual pair $(Y,X)$ for $\cK$ and an arbitrary element 
$\varphi \in \cK$, consider the maps: 
\begin{align*}
X\rightarrow Y \; , \; x\mapsto {}^-x=\vee \set{y\in Y}{\bracket{y}{x}\leq \varphi}
\; , \\ 
Y\rightarrow X \; , \; y\mapsto y^-=\vee \set{x\in X}{\bracket{y}{x}\leq \varphi}
\; .
\end{align*} 
Note that if the dual pair for $\opp{\cK}$ of Theorem~\ref{OppTheo} is used, 
then the suprema in the definition above are actually infima 
(in terms of the order of $\cK$) and the inequalities are reversed.  
The elements of $X$ and $Y$ of the form $y^-$ and ${}^-x$, 
respectively, are called {\em closed}. 
For $\lambda \in \cK$, we define 
\[
\lambda ^-:=\lambda \backslash \varphi=\vee \set{\mu \in \cK}{\lambda \mu\leq \varphi} 
\; \makebox{ and } \;
{}^-\lambda :=\varphi/ \lambda=\vee \set{\mu \in \cK}{\mu \lambda \leq \varphi} \; .
\] 
A complete idempotent semiring $\cK$ is 
{\em reflexive} if there exists $\varphi \in \cK$ such that 
${}^-(\lambda^-)=\lambda$ and $({}^-\lambda)^-=\lambda$ 
for all $\lambda \in \cK$. 
We shall need the following proposition. 

\begin{proposition}[{\cite[Prop.~32 and Prop.~27]{cgq02}}]\label{propXclosed}
If $\cK$ is reflexive and $(Y,X)$ is a pre-dual pair for which 
$Y$ separates $X$, then, all the elements of $X$ are closed. Moreover, 
for all $x,u\in X$ it is satisfied:
\begin{align}\label{prop1} 
u \backslash x=\bracket{{}^-x}{u}^- \; . 
\end{align}  
\end{proposition}

\begin{example}
The completed max-plus semiring $\rmaxb $ is reflexive. 
If we take $\varphi =0$, then ${}^-\lambda=\lambda^-=-\lambda $ 
for all $\lambda \in \rmaxb$. If $\rmaxb^n$ is  
equipped with the bracket $\bracket{y}{x}:=\oplus_i y_i x_i$, 
then~\eqref{prop1} becomes $u \backslash x=-(\oplus_i u_i (-x_i))$ 
for all $u,x\in \rmaxb^n$.
\end{example}

The results above are related to the max-plus analogues of the 
representation theorem for linear forms and the analytic form 
of the Hahn-Banach theorem of~\cite{cgq02} which extend the 
corresponding results of~\cite{litvinov00}.   
For a detailed presentation and 
examples of pre-dual pairs and reflexive semirings, 
we refer the reader to~\cite{cgq02}. 

\section{Separation Theorems}\label{SecSepTheo}

Given a right $\cK$-semimodule $X$, a subset $W$ of $ X^2$ 
is called a {\em pre-congruence} if it is a subsemimodule such that 
\[
(f,f)\in W \; , \; \forall f\in X
\]
and 
\[
(f,g)\in W\; ,\; (g,h)\in W \implies (f,h)\in W \; .
\]
If in addition a pre-congruence $W$ verifies 
\[
(f,g)\in W \implies (g,f)\in W \; ,
\]
then $W$ is said to be a {\em congruence}. In other words, 
a pre-congruence (resp. congruence) is a pre-order (resp. equivalence) 
relation on $X$ which has a semimodule 
structure when it is thought of as a subset of the semimodule $X^2$.  
A pre-congruence $W$ which satisfies the property
\begin{align}\label{Def1Polar}
f\leq g \implies (f,g)\in W \; ,
\end{align}
is called a {\em polar cone}. 

\begin{remark}
When $W$ is a pre-congruence, note that~\eqref{Def1Polar} 
is equivalent to
\begin{align}\label{Def2Polar}
f\leq g \; , \;  (g,h)\in W \implies (f,h)\in W \; .
\end{align}
Indeed, clearly~\eqref{Def2Polar} implies~\eqref{Def1Polar} because 
$(g,g)\in W$ for all $g\in X$. Conversely, 
since $W$ is a pre-congruence, it follows that 
$(f,g)\in W$ and $(g,h)\in W$ imply $(f,h)\in W$, 
and so~\eqref{Def1Polar} implies~\eqref{Def2Polar}. 
\end{remark}

The definition above is motivated by the classical notion of polar cones. 
Recall that if $V\subset \R^n$ is a (classical) convex cone, then its polar 
is defined as 
\[ 
\set{f\in \R^n}{\sum_i f_i x_i \leq 0, \forall x\in V} \; , 
\]
where $f_i x_i$ denotes the usual multiplication 
of the real numbers $f_i$ and $x_i$.   
The direct extension of this definition to $X=\rmax^n$, i.e. 
\[ 
\set{f\in \rmax^n}{\oplus_i f_i x_i \leq \zero, \forall x\in V}\; ,
\] 
where $f_i x_i$ denotes $f_i \otimes x_i=f_i + x_i$, 
the usual addition of real numbers, 
is not convenient because this set is usually trivial independently of 
$V\subset \rmax^n$. However, as it was mentioned in the introduction, 
the notion of polar can be extended to the max-plus algebra framework 
if we consider pairs of vectors instead of individual vectors, i.e. 
if we define the polar of $V$ by 
\begin{equation}\label{DefPolar} 
W:= \set{(f,g)\in (\rmax^n)^2}{\oplus_i f_i x_i \leq \oplus_j g_j x_j, 
\forall x\in V} \; . 
\end{equation} 
Observe that any set of this form is a polar cone 
of $X^2$ in the sense defined above. However, 
not all polar cones have this form (i.e. 
are the polar of a cone $V$), take for example 
$W=\set{(f,g)\in (\rmax)^2}{g=\zero \implies f=\zero}$. 

\begin{example}\label{ExamplePolar}
Let $\cK=\rmax$ and $X=\rmax^3$. Consider the semimodule $V\subset X$ 
generated by the vectors $a=(1,0,2)^T$, $b=(1,1,0)^T$ and $c=(2,4,0)^T$,  
i.e. the semimodule composed of all the max-plus linear combinations of 
these three vectors. This semimodule is represented by the bounded 
dark gray region of Figure~\ref{polar} together with the segments 
joining the points $a$ and $c$ to it. In this figure, 
every element of $V$ with finite entries is projected orthogonally 
onto any plane orthogonal to the main diagonal $(1,1,1)^T$ of $\R^3$. 
Note that vectors that are proportional in the max-plus sense 
are sent to the same point, so this gives a faithful image of $V$.

\begin{figure}
\begin{center}
\begin{tabular}[t]{cc}

\begin{picture}(0,0)%
\includegraphics{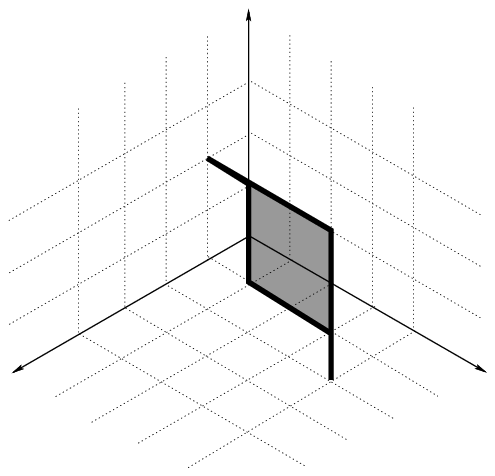}%
\end{picture}%
\setlength{\unitlength}{1302sp}%
\begingroup\makeatletter\ifx\SetFigFont\undefined%
\gdef\SetFigFont#1#2#3#4#5{%
  \reset@font\fontsize{#1}{#2pt}%
  \fontfamily{#3}\fontseries{#4}\fontshape{#5}%
  \selectfont}%
\fi\endgroup%
\begin{picture}(8424,7224)(1789,-7573)
\put(5026,-2836){\makebox(0,0)[lb]{\smash{{\SetFigFont{6}{7.2}{\rmdefault}{\mddefault}{\updefault}{\color[rgb]{0,0,0}$a$}%
}}}}
\put(9376,-5686){\makebox(0,0)[lb]{\smash{{\SetFigFont{6}{7.2}{\rmdefault}{\mddefault}{\updefault}{\color[rgb]{0,0,0}$x_2$}%
}}}}
\put(2401,-5611){\makebox(0,0)[lb]{\smash{{\SetFigFont{6}{7.2}{\rmdefault}{\mddefault}{\updefault}{\color[rgb]{0,0,0}$x_1$}%
}}}}
\put(6676,-4261){\makebox(0,0)[lb]{\smash{{\SetFigFont{6}{7.2}{\rmdefault}{\mddefault}{\updefault}{\color[rgb]{0,0,0}$V$}%
}}}}
\put(5476,-1036){\makebox(0,0)[lb]{\smash{{\SetFigFont{6}{7.2}{\rmdefault}{\mddefault}{\updefault}{\color[rgb]{0,0,0}$x_3$}%
}}}}
\put(7051,-6436){\makebox(0,0)[lb]{\smash{{\SetFigFont{6}{7.2}{\rmdefault}{\mddefault}{\updefault}{\color[rgb]{0,0,0}$c$}%
}}}}
\put(5476,-4711){\makebox(0,0)[lb]{\smash{{\SetFigFont{6}{7.2}{\rmdefault}{\mddefault}{\updefault}{\color[rgb]{0,0,0}$b$}%
}}}} 
\end{picture}%

& 

\begin{picture}(0,0)%
\includegraphics{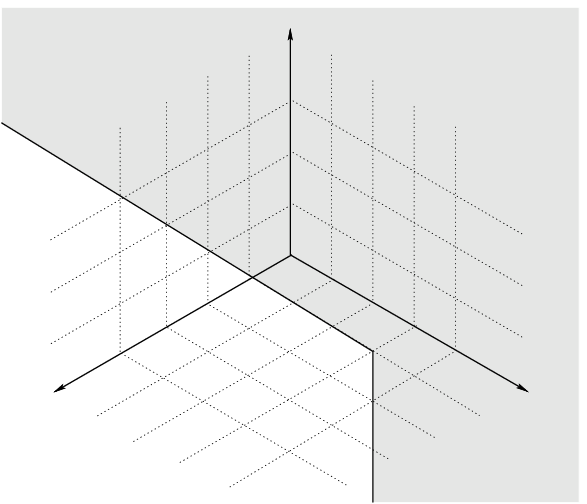}%
\end{picture}%
\setlength{\unitlength}{1302sp}%
\begingroup\makeatletter\ifx\SetFigFont\undefined%
\gdef\SetFigFont#1#2#3#4#5{%
  \reset@font\fontsize{#1}{#2pt}%
  \fontfamily{#3}\fontseries{#4}\fontshape{#5}%
  \selectfont}%
\fi\endgroup%
\begin{picture}(8434,7234)(1779,-7583)
\put(9376,-5686){\makebox(0,0)[lb]{\smash{{\SetFigFont{6}{7.2}{\rmdefault}{\mddefault}{\updefault}{\color[rgb]{0,0,0}$x_2$}%
}}}}
\put(2401,-5611){\makebox(0,0)[lb]{\smash{{\SetFigFont{6}{7.2}{\rmdefault}{\mddefault}{\updefault}{\color[rgb]{0,0,0}$x_1$}%
}}}}
\put(5476,-1036){\makebox(0,0)[lb]{\smash{{\SetFigFont{6}{7.2}{\rmdefault}{\mddefault}{\updefault}{\color[rgb]{0,0,0}$x_3$}%
}}}}
\put(6526,-2836){\makebox(0,0)[lb]{\smash{{\SetFigFont{6}{7.2}{\rmdefault}{\mddefault}{\updefault}{\color[rgb]{0,0,0}$2 x_1\leq x_2\oplus 3 x_3$}%
}}}}
\end{picture}% 

\end{tabular}

\begin{picture}(0,0)%
\includegraphics{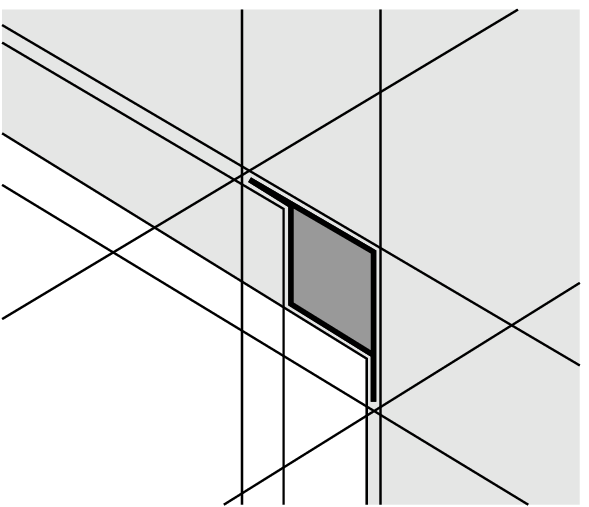}%
\end{picture}%
\setlength{\unitlength}{1302sp}%
\begingroup\makeatletter\ifx\SetFigFont\undefined%
\gdef\SetFigFont#1#2#3#4#5{%
  \reset@font\fontsize{#1}{#2pt}%
  \fontfamily{#3}\fontseries{#4}\fontshape{#5}%
  \selectfont}%
\fi\endgroup%
\begin{picture}(8466,7266)(1768,-7594)
\put(6526,-4336){\makebox(0,0)[lb]{\smash{{\SetFigFont{6}{7.2}{\rmdefault}{\mddefault}{\updefault}{\color[rgb]{0,0,0}$V$}%
}}}}
\end{picture}%

\end{center}
\caption{The semimodule $V$, a supporting half-space, 
and the polar of $V$.}
\label{polar} 
\end{figure}

If we define the semimodule $W\subset X^2$ by~\eqref{DefPolar},
then it can be checked that $W$ is a polar cone of $X^2$. 
In other words, this polar cone is composed of all the pairs 
of vectors $(f,g)$ such that the corresponding max-plus half-space 
$\set{x\in \rmax^3}{\oplus_i f_i x_i \leq \oplus_j g_j x_j}$ contains $V$. 
For example, if we take $f=(2,\zero,\zero)^T$ and $g=(\zero, 0, 3)^T$, 
then $(f,g)\in W$ because $a$, $b$, $c$ and hence $V$ 
are contained in the half-space 
$\set{x\in \rmax^3}{2x_1 \leq x_2 \oplus 3x_3}$, which is 
represented by the unbounded light gray region of Figure~\ref{polar}. 

In this case there exist eight max-plus linear inequalities 
$\oplus_i f_i x_i \leq \oplus_j g_j x_j$ which satisfy that $V$ 
is the intersection of the corresponding half-spaces. These inequalities, 
which are associated with elements of $W$ as explained above, 
are the following: 
$x_2 \leq 2 x_1$, 
$x_2 \leq 4 x_3$, 
$x_3 \leq 2 x_2$, 
$x_3 \leq 1 x_1$, 
$x_1 \leq 1 x_2$, 
$x_1 \leq 2 x_3$, 
$1 x_1 \leq 1 x_2 \oplus x_3$ and 
$2 x_1 \leq  x_2 \oplus 3 x_3$. 
The boundaries of the corresponding eight half-spaces have been represented in 
Figure~\ref{polar}. Only the half-space corresponding to the last 
inequality  has been depicted completely (note that this inequality 
is the one associated with the pair of vectors $(f,g)\in W$ 
considered above). 
\end{example}

Since the purpose of this section is to establish separation theorems 
for complete congruences and polar cones, in the rest of this section 
we will assume that $\cK$ is a complete idempotent 
semiring and that $X$ is a complete right $\cK$-semimodule. 

Given a complete pre-congruence $W\subset X^2$, for $g\in X$, we set
\begin{align}
\eme{g}:= \vee \set{f\in X}{(f,g)\in W} \; .
\end{align}
Observe that $(\eme{g},g)\in W$ because $W$ is complete. Moreover,  
since $(g,g)\in W$, it follows that $g\leq \eme{g}$ for all $g\in X$. 

When $W$ is a congruence, 
$\eme{g}$ is nothing but the maximal element in the 
equivalence class of $g$ modulo $W$. Therefore, in this 
case we have the following obvious properties, which hold
for all $g,f\in X$,
\begin{align}
g\leq \eme{g} = \eme{\eme{g}} \; ,\label{e-ob2}\\
(f,g)\in W \iff \eme{f}=\eme{g}\label{e-ob3}
\; .
\end{align}

We shall use the fact that
\begin{align}
f\leq g\implies \eme{f} \leq \eme{g} \; .\label{e-barinc}
\end{align}
Indeed, if $f\leq g$ then $g=f\oplus g$, 
and if $h\in X$ is such that $(h,f)\in W$,
we have $(h\oplus g,g)=(h\oplus g, f\oplus g)=(h,f)\oplus (g,g)\in W$ because
$W$ is a pre-congruence. It follows that $\eme{g}\geq h\oplus g\geq h$.
Since this holds for all $h\in X$ such that $(h,f)\in W$,
we deduce that $\eme{g}\geq \eme{f}$, which shows~\eqref{e-barinc}.

\begin{lemma}\label{lem-1}
Let $W\subset X^2$ be a complete pre-congruence. Then,
\begin{align*}
(f,g)\in W \implies g\backslash \eme{h}\leq f \backslash \eme{h} \label{e-sym1}
\end{align*}
for all $h\in X$.
\end{lemma}

\begin{proof}
Let $\lambda:=g\backslash \eme{h}$,
so that $\eme{h}\geq g\lambda $,
or equivalently $\eme{h} =\eme{h} \oplus g\lambda $.
Since $W$ is a pre-congruence, from  
$(\eme{h}\oplus f\lambda ,\eme{h})=
(\eme{h} \oplus f\lambda ,\eme{h}\oplus g\lambda )= 
(\eme{h},\eme{h}) \oplus (f,g)\lambda \in W$ and 
$(\eme{h},h)\in W$, we deduce that 
$(\eme{h} \oplus f\lambda,h)\in W$
and so $f\lambda \leq \eme{h}\oplus f\lambda \leq \eme{h}$. Hence, 
$f\backslash \eme{h} \geq \lambda$, i.e. 
$f\backslash \eme{h}\geq g\backslash \eme{h}$.
\end{proof} 

Now it is possible to prove the following separation theorem for complete 
polar cones.

\begin{theorem}\label{th-1}
Let $W\subset X^2$ be a complete polar cone. Assume that $s,t\in X$ 
are such that $(s,t)\not \in W$. Then, there exists $x\in X$ such that 
\[
(f,g)\in W \implies g\backslash x\leq f\backslash x
\]
and 
\[
t\backslash x \not \leq s\backslash x \; .
\]
\end{theorem}
\begin{remark}
In other words, Theorem~\ref{th-1} says that the set 
\[ 
M:=\set{(f,g)\in X^2}{g\backslash x\leq f\backslash x} 
\]
{\em separates} the complete polar cone $W$ and the point $(s,t)$: 
$W \subset M$ and $(s,t)\not \in M$. 
\end{remark}
\begin{proof}[Proof of Theorem~\ref{th-1}]
We will show that the assertion of the theorem holds with $x:=\eme{t}$. 

In the first place, note that if $x=\eme{t}$, then by Lemma~\ref{lem-1} 
the first assertion of the theorem is satisfied. 
Assume that $t\backslash \eme{t} \leq s\backslash \eme{t} $. Then, we get 
$s(t\backslash \eme{t})\leq s(s\backslash \eme{t})\leq \eme{t}$. 
Since $\eme{t}\geq t$, we have $t \backslash \eme{t} \geq e$, and so
$s\leq \eme{t}$. Thus, from $s\leq \eme{t}$ and $(\eme{t},t)\in W$, 
it follows that $(s,t)\in W$ because $W$ is a polar cone, 
which contradicts our assumption. 
\end{proof}

In the case of complete congruences, we have the following separation theorem.

\begin{theorem}\label{th-2}
Let $W\subset X^2$ be a complete congruence. 
Assume that $s,t\in X$ are such that $(s,t)\not \in W$. 
Then, there exists $x\in X$ such that 
\[
(f,g)\in W \implies f \backslash x = g \backslash x
\]
and 
\[
s \backslash x\neq t \backslash x \; .
\]
\end{theorem}

\begin{proof}
Since $(s,t)\not \in W$, by~\eqref{e-ob3} we have $\eme{s}\neq \eme{t}$, 
hence, $\eme{s}\not\leq \eme{t}$ and/or $\eme{t}\not\leq \eme{s}$. 
Assume for instance that $\eme{t}\not\leq \eme{s}$. 
Then, we claim that
\[
s \backslash \eme{s} \neq t\backslash \eme{s} \; .
\]
Indeed, if $s \backslash \eme{s} = t\backslash \eme{s}$, 
we get $t(s \backslash \eme{s})= t(t\backslash \eme{s})\leq \eme{s}$. 
Since $\eme{s}\geq s$,
we have $s \backslash \eme{s} \geq e$, and so
$t\leq \eme{s}$. From~\eqref{e-ob2} and~\eqref{e-barinc}, it follows 
that $\eme{t}\leq \eme{\eme{s}}=\eme{s}$,
which contradicts our assumption, so this proves our claim. 
Finally, as $(f,g)\in W$ implies that $(g,f)\in W$ 
because $W$ is a congruence, by Lemma~\ref{lem-1}
the assertion of the theorem holds with $x:=\eme{s}$.
\end{proof}

To recover separation theorems with the usual scalar product form, 
we need to consider reflexive semirings.

\begin{corollary}\label{CorSepRefCon}
Let $W\subset X^2$ be a complete congruence 
and $(Y,X)$ be a pre-dual pair for a 
reflexive semiring $\cK$ such that $Y$ separates $X$. 
If $(s,t)\not \in W$, then, there exists $y\in Y$ such that
\[
(f,g)\in W \implies \bracket{y}{f} =\bracket{y}{g} 
\]
and 
\[
\bracket{y}{s}\neq \bracket{y}{t} \; .
\]
\end{corollary}

\begin{proof}
By Theorem~\ref{th-2} and Proposition~\ref{propXclosed} it follows that  
\[
(f,g)\in W \implies \bracket{{}^-x}{f}^- =\bracket{{}^-x}{g}^- 
\]
and 
\[
\bracket{{}^-x}{s}^-\neq \bracket{{}^-x}{t}^- \; , 
\]
for some $x\in X$. Since $\cK$ is reflexive, 
the map $\lambda \mapsto \lambda^-$ is injective, 
and thus we have
\[
(f,g)\in W \implies \bracket{{}^-x}{f} =\bracket{{}^-x}{g} 
\]
and 
\[
\bracket{{}^-x}{s}\neq \bracket{{}^-x}{t} \; .
\]
Therefore, the assertion of the theorem holds with $y:={}^-x$.
\end{proof}

The following corollary of Theorem~\ref{th-1} and 
Proposition~\ref{propXclosed} can be proved along similar lines, 
using the fact that 
\[
\lambda_1\leq \lambda_2 \implies {}^-\lambda_2 \leq {}^-\lambda_1 \; , \;
\lambda_2^-\leq \lambda_1^-
\]
for all $\lambda_1,\lambda_2 \in \cK$.  
 
\begin{corollary}\label{CorSepRefPol}
Let $W\subset X^2$ be a complete polar cone and 
$(Y,X)$ be a pre-dual pair for a reflexive semiring 
$\cK$ such that $Y$ separates $X$. 
If $(s,t)\not \in W$, then, there exists $y\in Y$ such that
\[
(f,g)\in W \implies \bracket{y}{f} \leq \bracket{y}{g} 
\]
and 
\[
\bracket{y}{s}\not \leq \bracket{y}{t} \; .
\]
\end{corollary}

Given a pre-dual pair $(Y,X)$, 
we define the following correspondences between 
subsemimodules of $X^2$ and $Y$: 
\begin{align*}
W\subset X^2\mapsto W^\diamond
:=\set{y\in Y}{\bracket{y}{f} \leq \bracket{y}{g} \; , 
\; \forall (f,g)\in W} \; , \\
V\subset Y \mapsto V^\circ 
:=\set{(f,g)\in X^2}{\bracket{y}{f} \leq \bracket{y}{g}\; , 
\; \forall y\in V} \; ,
\end{align*} 
and 
\begin{align*}
W\subset X^2\mapsto W^\top
:=\set{y\in Y}{\bracket{y}{f} = \bracket{y}{g} \; , \; \forall (f,g)\in W} \; , \\
V\subset Y \mapsto V^\bot 
:=\set{(f,g)\in X^2}{\bracket{y}{f} = \bracket{y}{g} \; , \; \forall y\in V} \; .
\end{align*} 
Then, taking $Y=\opp{X}$ and $\bracket{y}{x}=x\backslash y$, 
the universal separation theorem for complete semimodules 
of Section~\ref{SectPrel} (see also~\cite[Th.~13]{cgq02}) 
implies that the following statements are equivalent:
\begin{enumerate}
\item[(i)] $V\subset Y$ is a complete semimodule, 
\item[(ii)] $V=(V^\circ )^\diamond $,
\item[(iii)] $V=(V^\bot)^\top $.
\end{enumerate} 
Now, thanks to the previous separation theorems, 
we can prove the dual result. 

\begin{theorem}\label{ThBipolar}
Let $(Y,X)$ be a pre-dual pair which satisfies the following property: 
If $W\subset X^2$ is a complete polar cone (resp. complete 
congruence) and $(s,t)\not \in W$, 
there exists $y\in Y$ such that
\[
(f,g)\in W \implies \bracket{y}{f} \leq \bracket{y}{g} 
\]
and 
\[
\bracket{y}{s}\not \leq \bracket{y}{t} \; .
\] 
Then, a subsemimodule $W\subset X^2$ is a complete polar cone 
(resp. complete congruence) if, and only if, 
\begin{align*}
W=(W^\diamond )^\circ \; \makebox{(resp. $W=(W^\top)^\bot $)}\; . 
\end{align*}
\end{theorem}

\begin{proof}
We only prove the theorem for polar cones. The case of congruences is similar. 

Note that any set of the form $V^\circ $ is a complete polar cone. 
This follows from the fact that the map $x\mapsto \bracket{y}{x}$ 
is right linear and continuous for all $y\in Y$. 

Since the inclusion $W\subset (W^\diamond )^\circ $ is straightforward, 
it suffices to show that $(W^\diamond )^\circ \subset W$. Assume that 
$(s,t)\not \in W$. Then, there exists $y\in Y$ such that 
$\bracket{y}{s}\not \leq \bracket{y}{t}$ and 
$\bracket{y}{f}\leq \bracket{y}{g}$ for all $(f,g)\in W$. 
Since this means that $y\in W^\diamond $ but 
$\bracket{y}{s}\not \leq \bracket{y}{t}$, 
we conclude that $(s,t)\not \in (W^\diamond )^\circ $. 
\end{proof}

Note that, by Theorems~\ref{th-1} and~\ref{th-2}, the 
condition of the previous theorem is satisfied when 
$Y=\opp{X}$ and $\bracket{y}{x}=x\backslash y$, 
and, by Corollaries~\ref{CorSepRefCon} and~\ref{CorSepRefPol}, 
it is also satisfied when $(Y,X)$ is a pre-dual pair for a 
reflexive semiring $\cK$ such that $Y$ separates $X$.

For any subset $U$ of $X^2$, we denote by $\cpol{U}$ 
(resp. $\ccong{U}$) the smallest complete polar 
cone (resp. complete congruence) containing it. 
Then, we have the following corollary. 

\begin{corollary}
Let $(Y,X)$ be a pre-dual pair which satisfies the condition 
in Theorem~\ref{ThBipolar}. Then, for any subset 
$U$ of $X^2$ we have $\cpol{U}=(U^\diamond )^\circ $ 
(resp. $\ccong{U}=(U^\top)^\bot $).
\end{corollary}

\begin{proof}
Since any set of the form $V^\circ $ is a complete polar cone, 
it follows that $(U^\diamond )^\circ $ is a complete polar cone 
which clearly contains $U$. 

Let $W$ be a complete polar cone containing $U$. Then, 
\[ 
U\subset W \implies W^\diamond \subset U^\diamond 
\implies (U^\diamond )^\circ \subset (W^\diamond )^\circ \; ,
\] 
and by Theorem~\ref{ThBipolar} we have 
$(W^\diamond )^\circ=W$. 
Therefore, $(U^\diamond )^\circ$ is the smallest complete polar cone 
containing $U$. 

The case of complete congruences can be proved along the same lines.
\end{proof}

\begin{example}\label{Illust}
When $\cK=\rmaxb$ and $X=\rmaxb^n$, note that if we consider 
the bracket $\bracket{y}{x}:=\oplus_i y_i x_i$, 
Corollary~\ref{CorSepRefPol} can be stated as follows: 
If $W\subset (\rmaxb^n)^2$ 
is a complete polar cone and $s,t\in \rmaxb^n$ 
are such that $(s,t)\not \in W$, 
then, there exists $y\in \rmaxb^n$ such that 
\[
(f,g)\in W\implies \oplus_i y_i f_i \leq \oplus_j y_j g_j
\]
and 
\[
\oplus_i y_i s_i \not \leq \oplus_j y_j t_j \; .
\] 
For instance, consider the polar cone $W:=V^\circ $, 
where $V$ is the subsemimodule of $\rmaxb^3$ 
generated by the vectors $a=(2,0,3)^T$, $b=(2,1,0)^T$ and 
$c=(2,4,0)^T$. This semimodule  
is represented by the bounded dark gray region of Figure~\ref{figuresemi} 
together with the segment joining the point $a=(2,0,3)^T$ to it. 

\begin{figure}
\begin{center}

\begin{picture}(0,0)%
\includegraphics{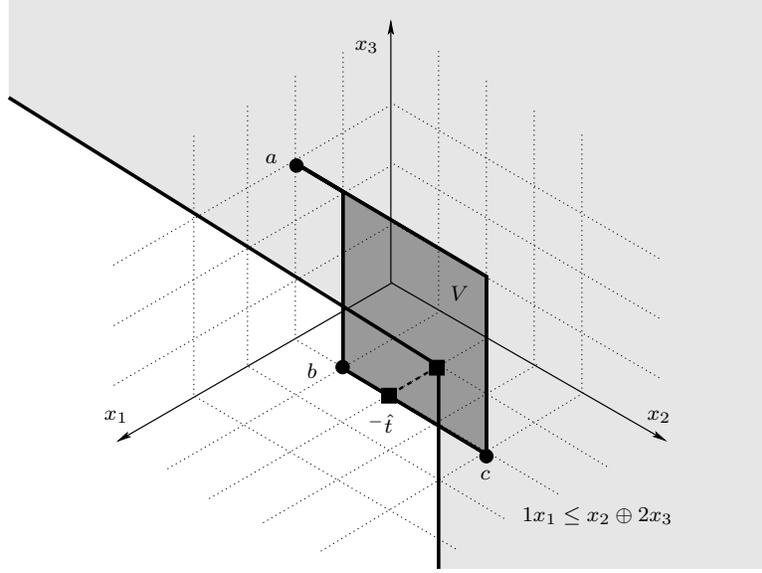}%
\end{picture}%
\setlength{\unitlength}{1973sp}%
\begingroup\makeatletter\ifx\SetFigFont\undefined%
\gdef\SetFigFont#1#2#3#4#5{%
  \reset@font\fontsize{#1}{#2pt}%
  \fontfamily{#3}\fontseries{#4}\fontshape{#5}%
  \selectfont}%
\fi\endgroup%
\begin{picture}(9645,7245)(1157,-7605)
\put(2401,-5686){\makebox(0,0)[lb]{\smash{{\SetFigFont{8}{9.6}{\rmdefault}{\mddefault}{\updefault}{\color[rgb]{0,0,0}$x_1$}%
}}}}
\put(9226,-5686){\makebox(0,0)[lb]{\smash{{\SetFigFont{8}{9.6}{\rmdefault}{\mddefault}{\updefault}{\color[rgb]{0,0,0}$x_2$}%
}}}}
\put(5551,-1036){\makebox(0,0)[lb]{\smash{{\SetFigFont{8}{9.6}{\rmdefault}{\mddefault}{\updefault}{\color[rgb]{0,0,0}$x_3$}%
}}}}
\put(7126,-6436){\makebox(0,0)[lb]{\smash{{\SetFigFont{8}{9.6}{\rmdefault}{\mddefault}{\updefault}{\color[rgb]{0,0,0}$c$}%
}}}}
\put(4951,-5161){\makebox(0,0)[lb]{\smash{{\SetFigFont{8}{9.6}{\rmdefault}{\mddefault}{\updefault}{\color[rgb]{0,0,0}$b$}%
}}}}
\put(5716,-5851){\makebox(0,0)[lb]{\smash{{\SetFigFont{8}{9.6}{\rmdefault}{\mddefault}{\updefault}{\color[rgb]{0,0,0}${}^-\eme{t}$}%
}}}}
\put(6751,-4186){\makebox(0,0)[lb]{\smash{{\SetFigFont{8}{9.6}{\rmdefault}{\mddefault}{\updefault}{\color[rgb]{0,0,0}$V$}%
}}}}
\put(4426,-2461){\makebox(0,0)[lb]{\smash{{\SetFigFont{8}{9.6}{\rmdefault}{\mddefault}{\updefault}{\color[rgb]{0,0,0}$a$}%
}}}}
\put(7651,-6961){\makebox(0,0)[lb]{\smash{{\SetFigFont{8}{9.6}{\rmdefault}{\mddefault}{\updefault}{\color[rgb]{0,0,0}$1x_1\leq x_2\oplus 2x_3$}%
}}}}
\end{picture}%

\end{center}
\caption{Illustration of the separation theorem for complete polar cones.}
\label{figuresemi} 
\end{figure}

If we take $s=(1,\zero ,\zero)^T$ and $t=(\zero ,0,2)^T$, 
then $(s,t)\not \in W$ because the inequality 
$1 x_1\leq x_2 \oplus 2 x_3$ is not satisfied 
by the element $x=b=(2,1,0)^T\in V$. 
Let us compute the vector ${}^-\eme{t}$ which, 
according to Theorem~\ref{th-1} and Corollary~\ref{CorSepRefPol}, 
should separate $(s,t)$ from $W$. 
By definition, 
\[
\eme{t}=\vee \set{f\in \rmaxb^3}{(f,t)\in W} = 
\vee \set{f\in \rmaxb^3}{V \subset H_f}\; , 
\] 
where 
\[ 
H_f:=\set{x\in \rmaxb^3}{\oplus_i f_i x_i \leq \oplus_j t_j x_j} = 
\set{x\in \rmaxb^3}{f_1 x_1\oplus f_2 x_2 \oplus f_3 x_3 \leq x_2 \oplus 2 x_3}
.
\] 
Note that if we define $I:=\set{1\leq i \leq 3}{f_i > t_i}$ and 
$J:= \left\{ 1,2,3 \right\}\setminus I$, we have 
\[ 
H_f=\set{x\in \rmaxb^3}{\oplus_{i\in I} f_i x_i \leq \oplus_{j\in J} t_j x_j}
\; . 
\] 
Then, we can have the following cases depending on the vector $f$: 
\begin{itemize}
\item[ - ] If $f_2\leq 0$ and $f_3 \leq 2$, then 
$H_f=\set{x\in \rmaxb^3}{f_1 x_1 \leq x_2 \oplus 2x_3}$. 
Note that the inequality $f_1 x_1 \leq x_2 \oplus 2x_3$  is 
satisfied by the generators $a$, $b$ and $c$ of $V$ if, 
and only if, $f_1\leq 0$. 
Therefore, since $H_f$ is a semimodule, it follows that 
$V\subset H_f$ if, and only if, $f_1\leq 0$, 
which shows that $\eme{t}\geq (0,0,2)^T$. 
\item[ - ] If $f_2 > 0$ and $f_3 \leq 2$, then 
$H_f  =  \set{x\in \rmaxb^3}{f_1 x_1\oplus f_2 x_2 \leq 2 x_3}$. 
Note that $(2,4,0)^T $ never belongs to $H_f$ because this 
would imply that $f_2 \leq -2$. Then, in this case, 
$V$ is never contained in $H_f$  because the vector 
$c$ does not belong to $H_f$.  
\item[ - ] If $f_2 \leq 0$ and $f_3 > 2$, then 
$H_f=\set{x\in \rmaxb^3}{f_1 x_1\oplus f_3 x_3 \leq x_2}$. 
Note that $(2,0,3)^T \not \in H_f$ because $f_3 > 2$. 
Therefore, like in the previous case, the set $H_f$ 
never contains $V$ because in particular it does not contain the vector $a$. 
\item[ - ]  If $f_2 > 0$ and $f_3 > 2$,  
the elements of $H_f$ cannot have finite second and third entries. 
Then, in this case, $V$ is never contained in $H_f$.
\end{itemize}
Therefore, we conclude that $\eme{t}=(0,0,2)^T$ and so 
${}^-\eme{t}=-\eme{t}=(0,0,-2)^T$. Since 
${}^-\eme{t}=(-2)b\oplus(-4) c \in V$, it follows that 
$\bracket{{}^-\eme{t}}{f} \leq \bracket{{}^-\eme{t}}{g}$ for all 
$(f,g)\in W$. However,  
$\bracket{{}^-\eme{t}}{s} = 1 > 0 =\bracket{{}^-\eme{t}}{t}$ 
and thus ${}^-\eme{t}$ separates $(s,t)$ from $W$. 
\end{example}

A natural way to define a congruence $W$ is to take a continuous right linear
map $F$ from $X$ to a complete right $\cK$-semimodule $Z$ and to define 
$W=\ker F:=\set{(u,v)\in X^2}{F(u)=F(v)}$. Then, since $F$ is continuous, 
$W$ is complete and it satisfies
\[
\eme{u}=\vee \set{v\in X}{F(v)=F(u)}=
\vee \set{v\in X}{F(v)\leq F(u)}=
F^\sharp (F (u)) \; ,
\]
for all $u\in X$. Conversely, it can be checked that every
complete congruence $W$ arises in this way.
It suffices to take for $F$ the canonical morphism
from $X$ to its quotient by the equivalence relation $W$. The details
are left to the reader. 

\begin{remark}
A basic special case is when $X=\rmaxb^n$ and $F$ is a continuous
(right) linear map from $X$ to $\rmaxb^p$, so that $F$ can be represented
by a $p\times n$ matrix $A=(a_{ij})$ with entries in $\rmaxb$,
meaning that the $i$-th coordinate of $F$ is given by 
$F_i(x)=\max_j (a_{ij}+x_j)$. 
The residual of $F$ is the min-plus linear map represented
by the matrix $A$ transposed and changed of sign $(-a_{ji})$, 
$F^\sharp_j(y)=\min_i(-a_{ij}+y_i)$, 
with the convention that $+\infty$ is absorbing for the (usual) addition. 
Hence, the maximal representative $\eme{u}=F^\sharp\circ F(u)$ 
is determined by a simple ``min-max'' combinatorial formula. 
\end{remark}

\section{Separation theorems for the max-plus semiring}

In the preceding section, the semimodules, polar cones, and congruences 
were required to be complete, which is the same
as being closed in the Scott topology. This topology is 
adapted when the underlying semiring is
the completed max-plus semiring, $\rmaxb$ 
(we refer the reader to~\cite{akian,akiansinger} for a discussion of this topology within the max-plus setting). However,
in many applications, the semiring of interest
is rather the non completed max-plus semiring $\rmax$, and the topology
of choice is the standard one, which can be defined
by the metric $d(a,b):=|\exp(a)-\exp(b)|$. Thus, instead
of complete pre-congruences over $\rmaxb^n$, we now deal with
pre-congruences over $\rmax^n$ that are closed
(in the induced product topology), and we would like
to find separation theorems in the spirit of Theorem~\ref{th-1} 
involving closed separating sets.
We may of course consider the restriction to $(\rmax^n)^2$ of 
the separating set constructed in the proof of this theorem, i.e. 
\[ \set{(f,g)\in (\rmax^n)^2}{g\backslash \eme{t}\leq f\backslash \eme{t}} \; ,
\]
but it can be checked that this set is not closed as soon as $\eme{t}$ has
a $-\infty$ coordinate (see the example below). 
In this section, we apply a perturbation technique
to derive separation theorems which are appropriate for closed polar
cones and congruences over $\rmax^n$. 

\begin{example}
Consider again the polar cone $W=V^\circ $ and the pair 
of vectors $(s,t)$ of Example~\ref{Illust}, but assume that we add 
the vector $(\unit, \zero ,\zero )^T$ to the set of generators of $V$. 
Then, it can be checked that with this modification 
we have $\eme{t}=t=(\zero , 0,2)^T$. Therefore, the restriction to 
$(\rmax^3)^2$ of the set that separates $(s,t)$ from $W$ (given by 
Theorem~\ref{th-1}) is 
\begin{align*} 
& \set{(f,g)\in (\rmax^3)^2}
{(+\infty) f_1\oplus f_2 \oplus (-2) f_3 \leq 
(+\infty) g_1\oplus g_2 \oplus (-2) g_3} = 
\\
& \set{(f,g)\in (\rmax^3)^2}{g_1\neq \zero} \; \cup \\
& \; \; \; \; \;  \; \; \; \; \; \; \; \; 
\set{(f,g)\in (\rmax^3)^2}
{f_1 = g_1 = \zero , f_2 \oplus (-2) f_3 \leq g_2 \oplus (-2) g_3 }
\end{align*}
which is not a closed subset of $(\rmax^3)^2$.
\end{example}

Inspired by the previous section, 
if $W\subset (\rmax^n)^2$ is a closed pre-congruence, 
for $g\in \rmax^n$ we set
\begin{align}
\emb{g}:= \sup\set{f\in \rmax^n}{(f,g)\in W}\in \rmaxb^n \; .
\end{align}
Unlike the case of complete pre-congruences, 
observe that $\emb{g}$ may have some entries equal to $+\infty$, 
so $(\emb{g},g)$ need not belong to $W$. 
However, since $(g,g)\in W$, we still 
have $g\leq \emb{g}$ for all $g\in \rmax^n$. 
We shall use the fact that
\begin{align}
g\leq f\implies \emb{g} \leq \emb{f} \; ,\label{e-barinc-Rmax}
\end{align}
which can be proved as in the case of complete pre-congruences. 

Due to the fact that $(\emb{g},g)$ does not necessarily belong to $W$, 
we cannot use the vector $\emb{g}$ in 
the same way we did it before. Therefore, 
we need a perturbation argument in which the following simple 
construction will be very important. 
For $z\in \rmaxb^n$ and $\beta \in \rmax$, 
we define the vector $z^\beta \in \rmax^n$ by: 
\[
z^\beta_i=\left\{
\begin{array}{cl}
\beta  &  \makebox{ if } z_i=+\infty \; , \\
z_i & \makebox{ otherwise.}
\end{array}
\right.  
\]
In other words, $z^\beta$ is obtained from $z$ by replacing its  
$+\infty$ entries by $\beta$. 
We denote by $\uvector_i$ the $i$-th unit vector, i.e. the 
vector defined by $(\uvector_i)_i:=\unit$ and $(\uvector_i)_j:=\zero$ 
if $j\neq i$. We shall need the following lemma. 

\begin{lemma}\label{simpleprop}
Let $W\subset (\rmax^n)^2$ be a congruence and 
$f,g\in \rmax^n$ be such that $(f,g)\in W$. Then, 
$(g,g\oplus \uvector_i\gamma )\in W$ for all $\gamma \leq f_i$. 
\end{lemma}
\begin{proof}
If $\gamma\leq f_i$, since $W$ is a congruence and $(f,g)\in W$, we have 
$(f, g\oplus \uvector_i\gamma)=
(f\oplus \uvector_i\gamma, g\oplus \uvector_i\gamma)=
(f,g)\oplus (\uvector_i, \uvector_i)\gamma \in W$. Hence, 
$(g,g\oplus \uvector_i\gamma)\in W $ because $(g,f)\in W$,
$(f, g\oplus \uvector_i\gamma) \in W$ and $W$ is a congruence. 
\end{proof}

Then we have the following corollary.

\begin{corollary}\label{coro-s1}
Let $W\subset (\rmax^n)^2$ be a closed congruence or a closed 
polar cone. Then, for all $g\in \rmax^n$, 
if $\beta\geq \max_i (g_i)$ we have $({\emb g}^\beta, g)\in W$.
\end{corollary}

\begin{proof} 
Let us first consider the case where $W$ is a closed congruence. 
If ${\emb g}_i=+\infty$, there exists $f\in \rmax^n$ 
such that $(f,g)\in W$ and $f_i> \beta$. Then, by Lemma~\ref{simpleprop} 
we have $(g, g\oplus \uvector_i\beta) \in W$. 

Now assume that $\emb{g}_i < +\infty$. Then, if $\gamma <{\emb{g}}_i$, 
there exists $f\in \rmax^n$ such that $(f,g)\in W$ and 
$f_i> \gamma$. Thus, by Lemma~\ref{simpleprop} we have 
$(g,g\oplus \uvector_i \gamma )\in W$. 
Since this holds for all $\gamma\in \rmax$ 
such that $\gamma <{\emb{g}}_i$, 
it follows that $(g,g\oplus \uvector_i \emb{g}_i)\in W$ 
because $W$ is closed. 

Consequently, as $W$ is a semimodule, we have 
\begin{align*}
({\emb{g}}^\beta,g) & = 
((\oplus_{\set{i}{\emb{g}_i <+\infty}} g\oplus \uvector_i \emb{g}_i )\oplus 
(\oplus_{\set{i}{\emb{g}_i = +\infty}} g\oplus \uvector_i \beta ),g) \\ 
& =  (\oplus_{\set{i}{\emb{g}_i <+\infty}}(g\oplus \uvector_i \emb{g}_i,g ))
\oplus (\oplus_{\set{i}{\emb{g}_i = +\infty}}(g\oplus \uvector_i \beta,g )) 
\in W  \; .
\end{align*} 

The case of polar cones is easier (indeed the assertion 
holds for any $\beta \in \rmax$). If $\emb{g}_i=+\infty$, 
there exists $f\in \rmax^n$ such that $(f,g)\in W$ and $f_i> \beta$. Then, 
as $\uvector_i \beta \leq f$ and $(f,g)\in W$, 
it follows that $(\uvector_i \beta ,g)\in W$, 
because $W$ is a polar cone. 

Assume now that $\emb{g}_i < +\infty$. Then, if $\alpha < \emb{g}_i$, 
there exists $f\in \rmax^n$ such that $(f,g)\in W$ and 
$f_i> \alpha$. Therefore, as $\uvector_i \alpha \leq f$ and $(f,g)\in W$, 
we have $(\uvector_i \alpha ,g)\in W$. 
Since this holds for all $\alpha \in \rmax$ such that 
$\alpha < \emb{g}_i$, we deduce that $(\uvector_i \emb{g}_i ,g)\in W$, 
because $W$ is closed. 

Finally, as $W$ is a semimodule, we get 
\begin{align*}
({\emb{g}}^\beta,g) & = 
((\oplus_{\set{i}{\emb{g}_i <+\infty}} \uvector_i \emb{g}_i  )
\oplus (\oplus_{\set{i}{\emb{g}_i = +\infty}} \uvector_i \beta ) ,g) \\ 
& =  ((\oplus_{\set{i}{\emb{g}_i <+\infty}} (\uvector_i \emb{g}_i ,g)  )
\oplus (\oplus_{\set{i}{\emb{g}_i = +\infty}} (\uvector_i \beta ,g) )  
\in W  \; , 
\end{align*}  
which completes the proof of the corollary.
\end{proof}
  
Like in the case of complete pre-congruences, we have:

\begin{lemma}\label{lemma-1-Rmax}
Let $W\subset (\rmax^n)^2$ be a closed congruence or a closed polar cone. 
Then, 
\[
(f,g) \in W \implies g\backslash \emb{h} \leq f \backslash \emb{h}  
\]
for all $h\in \rmax^n$.
\end{lemma}

\begin{proof}
When $g\backslash \emb{h}=-\infty$ the assertion is obvious, so assume that  
$g\backslash \emb{h}\not =-\infty$. Let $\alpha <g\backslash \emb{h}$, 
so that $\emb{h} \geq g\alpha$. Then, if 
$\beta \geq \max_i (g_i\alpha \oplus h_i)$, we have 
$g\alpha \leq \emb{h}^\beta $, and by Corollary~\ref{coro-s1} we know that  
$(\emb{h}^\beta ,h)\in W$. 
Since $W$ is a pre-congruence, from 
$(\emb{h}^\beta \oplus f\alpha ,\emb{h}^\beta ) =
(\emb{h}^\beta \oplus f\alpha ,\emb{h}^\beta \oplus g\alpha ) = 
(\emb{h}^\beta ,\emb{h}^\beta )\oplus (f,g) \alpha \in W$ and 
$(\emb{h}^\beta ,h)\in W$, it follows that 
$(\emb{h}^\beta  \oplus f\alpha , h)\in W $,
and so, $f\alpha \leq \emb{h}^\beta \oplus f\alpha \leq \emb{h}$. 
Therefore, we have $f\backslash \emb{h} \geq \alpha$. Since this holds for 
all $\alpha \in \rmax$ such that $\alpha < g \backslash \emb{h}$, 
we conclude that $g \backslash \emb{h} \leq f \backslash \emb{h}$.   
\end{proof}

The following lemmas will be useful to prove the 
separation theorems for the max-plus semiring.

\begin{lemma}\label{lemma-2-Rmax}
Let $W\subset (\rmax^n)^2$ be a closed pre-congruence. 
Assume that $s,t,h\in \rmax^n$ are such that 
$t\leq h$ and $s\not \leq \emb{h}$. 
Then, we have $t\backslash \emb{h}\not \leq s\backslash \emb{h} $. 
\end{lemma}

\begin{proof}
Assume that $t\backslash \emb{h} \leq s\backslash \emb{h} $. Then, we get 
$s (t\backslash \emb{h}) \leq s (s\backslash \emb{h}) \leq \emb{h}$. 
Since $t\leq h$, by~\eqref{e-barinc-Rmax} we have $t\leq \emb{t}\leq \emb{h}$, 
and so $e\leq t\backslash \emb{h}$. Therefore, it follows that 
$s=s e\leq s (t\backslash \emb{h}) \leq \emb{h}$, 
which is a contradiction. 
\end{proof}

\begin{lemma}\label{propimpor}
Let $W\subset (\rmax^n)^2$ be a closed congruence or 
a closed polar cone. Assume that $s,t\in \rmax^n$ 
are such that $s\not \leq \emb{t}$. Then, 
there exists $h\in \rmax^n$ with finite entries 
such that $t\leq h$ and $s\not\leq \emb{h}$.
\end{lemma}

\begin{proof}
Assume that $s\leq \emb{h}$ for all $h\in \rmax^n$ 
with finite entries that satisfy $t\leq h$. Let 
$\{ h_k \}_{k\in \N}\subset \rmax^n$ be a decreasing sequence of vectors
with finite entries such that $\lim_{k\rightarrow \infty} h_k = t$. 
Then, as by~\eqref{e-barinc-Rmax} the sequence 
$\{\emb{h}_k \}_{k\in \N}\subset \rmaxb^n$ is also decreasing, 
there exists $z\in \rmaxb^n$ such that 
$\lim_{k\rightarrow \infty}\emb{h}_k = z$. Since $s\leq \emb{h}_k$ 
for all $k\in \N$, we have $s\leq z$. Note that we can assume, 
without loss of generality, that 
$(\emb{h}_k)_i=+\infty \iff z_i=+\infty$. 
If $\lambda \geq \max_i (h_1)_i$, then by Corollary~\ref{coro-s1} we know 
that $((\emb h_k)^\lambda , h_k)\in W$ for all $k\in \N$. Therefore, we get
$(z^\lambda,t)=\lim_{k\rightarrow \infty}((\emb{h}_k)^\lambda ,h_k) \in W$, 
because $W$ is closed. Since this holds for all $\lambda \in \rmax$ such that 
$\lambda \geq \max_i (h_1)_i$, it follows that $\emb{t}\geq z \geq s$, 
which contradicts our assumption.  
\end{proof}

As a consequence of the previous results we obtain the separation theorem 
for closed polar cones.

\begin{theorem}\label{TSepPolRmax}
Let $W\subset (\rmax^n)^2$ be a closed polar cone. 
Assume that $s,t\in \rmax^n$ are such that $(s,t)\not \in W$. 
Then, there exists $y\in \rmax^n$ such that 
\[
(f,g)\in W \implies \oplus_i f_i y_i\leq \oplus_j g_j y_j
\]
and 
\[
\oplus_i s_i y_i\not \leq \oplus_j t_j y_j \; .
\]
\end{theorem}

\begin{proof}
In the first place, we claim that $s\not \leq \emb{t}$. Indeed, 
if $s\leq \emb{t}$, let $\beta\geq \max_i (s_i\oplus t_i)$. 
Then, $s\leq \emb{t}^\beta$ and by Corollary~\ref{coro-s1} we have 
$(\emb{t}^\beta ,t)\in W$. 
It follows that $(s,t)\in W$ because $W$ is a polar cone, 
which contradicts our assumption. This proves our claim.

Since $s\not \leq \emb{t}$ , according to Lemma~\ref{propimpor} 
there exists $h\in \rmax^n$ 
with finite entries such that $t\leq h$ and $s\not\leq \emb{h}$. 
Therefore, from Lemmas~\ref{lemma-1-Rmax} and~\ref{lemma-2-Rmax}, 
it follows that 
\[
(f,g)\in W \implies g\backslash \emb{h} \leq f\backslash \emb{h} 
\]
and 
\[
t\backslash \emb{h}\not \leq s\backslash \emb{h} \; . 
\]
Then, as $u \backslash \emb{h}=-(\oplus_i u_i (-\emb{h}_i))$ and 
$\emb{h}_i\geq h_i > -\infty$ for all $i$, 
the assertion of the theorem holds with $y:=-\emb{h}$.  
\end{proof}

In order to prove a separation theorem for closed congruences, 
we need the following result.

\begin{lemma}\label{LemmaProp}
Let $W\subset (\rmax^n)^2$ be a closed congruence. Then, 
for all $f,g\in \rmax^n$, the following properties are satisfied. 
\begin{align}\label{p1} 
(f,g)\in W \iff \emb{f} = \emb{g} \; ;\\ 
f\leq \emb{g} \implies \emb{f} \leq \emb{g} \; .\label{p2} 
\end{align}
\end{lemma}

\begin{proof}
\eqref{p1} If $(f,g) \in W$, we have  
\[
\set{h\in \rmax^n}{(h,f)\in W}=\set{h\in \rmax^n}{(h,g)\in W}
\]
because $W$ is a congruence. Therefore, $\emb{f}=\emb{g}$. 

Conversely, assume that $\emb{f}=\emb{g}$. 
If $\lambda \geq \max_i(f_i\oplus g_i)$, 
by Corollary~\ref{coro-s1} it follows that $(\emb{f}^\lambda , f) \in W$ and 
$(\emb{g}^\lambda , g)\in W$. Since $\emb{f}=\emb{g}$, 
we have $\emb{f}^\lambda =\emb{g}^\lambda$, 
and thus $(f,g)\in W$ because $W$ is a congruence.
 
\eqref{p2} Let $\lambda \geq \max_i (f_i\oplus g_i)$. Then, 
if $f\leq \emb{g}$ we have $f\leq {\emb{g}}^\lambda$, 
and by Corollary~\ref{coro-s1} we know that 
$({\emb {g}}^\lambda ,g)\in W$. Therefore, 
by~\eqref{e-barinc-Rmax} and~\eqref{p1}, 
it follows that $\emb{f}\leq \wemb{\emb{g}^\lambda}=\emb{g}$. 
\end{proof}

\begin{theorem}\label{TSepConRmax}
Let $W\subset (\rmax^n)^2$ be a closed congruence. Assume that 
$s,t\in \rmax^n$ are such that $(s,t) \not \in W$. Then, there exists 
$y\in \rmax^n$ such that 
\[
(f,g) \in W \implies \oplus_i f_i y_i = \oplus_j g_j y_j
\]
and 
\[
\oplus_i s_i y_i \neq \oplus_j t_j y_j \; .
\]
\end{theorem}

\begin{proof}
Since $(s,t)\not \in W$, by~\eqref{p1} we have $\emb{s}\neq \emb{t}$, 
hence, $\emb{s}\not\leq \emb{t}$ and/or $\emb{t}\not\leq \emb{s}$. 
Assume for instance that $\emb{s}\not\leq \emb{t}$. Then, 
from~\eqref{p2} it follows that $s\not\leq \emb{t}$. Therefore, 
according to Lemma~\ref{propimpor} there exists $h\in \rmax^n$ 
with finite entries such that $t\leq h$ and $s\not\leq \emb{h}$. 
Since $(f,g)\in W$ implies $(g,f)\in W$ because $W$ is a congruence, 
from Lemmas~\ref{lemma-1-Rmax} and~\ref{lemma-2-Rmax}, we have 
\[
(f,g) \in W \implies f\backslash \emb{h} = g\backslash \emb{h}
\]
and 
\[
t\backslash \emb{h} \neq s\backslash \emb{h}\; .
\]
Then, taking into account that 
$u \backslash \emb{h}=-(\oplus_i u_i (-\emb{h}_i))$ 
and $\emb{h}_i\geq h_i>-\infty$ for all $i$, 
the assertion of the theorem holds with $y:=-\emb{h}$.     
\end{proof}

Like in the previous section, consider the following 
correspondences between subsemimodules 
of $(\rmax^n)^2$ and $\rmax^n$: 
\begin{align*}
W\subset (\rmax^n)^2\mapsto W^\diamond
:=\set{x\in \rmax^n}{\oplus_i f_i x_i \leq \oplus_j g_j x_j \; ,\; 
\forall (f,g)\in W} \; , \\
V\subset \rmax^n \mapsto V^\circ 
:=\set{(f,g)\in (\rmax^n)^2}{\oplus_i f_i x_i \leq \oplus_j g_j x_j \; , \; 
\forall x\in V} \; ,
\end{align*} 
and 
\begin{align*}
W\subset(\rmax^n)^2\mapsto W^\top 
:=\set{x\in \rmax^n}{\oplus_i f_i x_i =\oplus_j g_j x_j \; , \; 
\forall (f,g)\in W} \; , \\
V\subset \rmax^n \mapsto V^\bot 
:=\set{(f,g)\in (\rmax^n)^2}{\oplus_i f_i x_i =\oplus_j g_j x_j \; , \; 
\forall x\in V} \; .
\end{align*}  
Then, by the separation theorem for closed semimodules 
(see~\cite[Th.~4]{zimmerman77},~\cite{shpiz}, 
see also~\cite[Th.~3.14]{cgqs04} for recent improvements), 
it follows that the following statements are equivalent: 
\begin{enumerate}
\item[(i)] $V\subset \rmax^n$ is a closed semimodule, 
\item[(ii)] $V=(V^\circ )^\diamond $,
\item[(iii)] $V=(V^\bot)^\top$.
\end{enumerate} 
Now, as a consequence of Theorems~\ref{TSepPolRmax} and~\ref{TSepConRmax}, 
we obtain the following dual result. We omit the proof because 
it is similar to that of Theorem~\ref{ThBipolar}.

\begin{theorem}\label{ThBipolarRmax}
A semimodule $W\subset (\rmax^n)^2$ is a closed polar cone 
(resp. closed congruence) if, and only if,  
\begin{align*}
W=(W^\diamond )^\circ \; \makebox{(resp. $W=(W^\top)^\bot $)}\; . 
\end{align*}
\end{theorem}

As a corollary of the previous theorem, it follows that 
for any subset $U$ of $(\rmax^n)^2$, 
$(U^\diamond )^\circ $ (resp. $(U^\top)^\bot$) 
is the smallest closed polar cone (resp. 
closed congruence) containing it. 

\begin{figure}
\begin{center}
\begin{tabular}[t]{cc}

\begin{picture}(0,0)%
\includegraphics{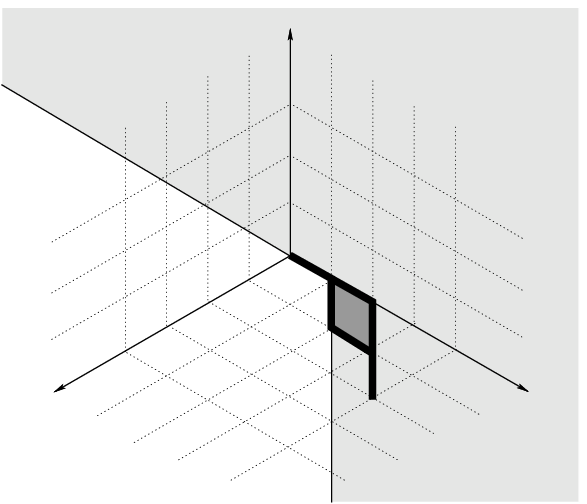}%
\end{picture}%
\setlength{\unitlength}{1302sp}%
\begingroup\makeatletter\ifx\SetFigFont\undefined%
\gdef\SetFigFont#1#2#3#4#5{%
  \reset@font\fontsize{#1}{#2pt}%
  \fontfamily{#3}\fontseries{#4}\fontshape{#5}%
  \selectfont}%
\fi\endgroup%
\begin{picture}(8444,7244)(1779,-7583)
\put(2401,-5611){\makebox(0,0)[lb]{\smash{{\SetFigFont{6}{7.2}{\rmdefault}{\mddefault}{\updefault}{\color[rgb]{0,0,0}$x_1$}%
}}}}
\put(6751,-4936){\makebox(0,0)[lb]{\smash{{\SetFigFont{6}{7.2}{\rmdefault}{\mddefault}{\updefault}{\color[rgb]{0,0,0}$V$}%
}}}}
\put(9226,-5686){\makebox(0,0)[lb]{\smash{{\SetFigFont{6}{7.2}{\rmdefault}{\mddefault}{\updefault}{\color[rgb]{0,0,0}$x_2$}%
}}}}
\put(6526,-1936){\makebox(0,0)[lb]{\smash{{\SetFigFont{6}{7.2}{\rmdefault}{\mddefault}{\updefault}{\color[rgb]{0,0,0}$1x_1\leq x_2\oplus 1x_3$}%
}}}}
\put(5401,-1036){\makebox(0,0)[lb]{\smash{{\SetFigFont{6}{7.2}{\rmdefault}{\mddefault}{\updefault}{\color[rgb]{0,0,0}$x_3$}%
}}}}
\put(6076,-3886){\makebox(0,0)[lb]{\smash{{\SetFigFont{6}{7.2}{\rmdefault}{\mddefault}{\updefault}{\color[rgb]{0,0,0}$a$}%
}}}}
\put(7426,-6136){\makebox(0,0)[lb]{\smash{{\SetFigFont{6}{7.2}{\rmdefault}{\mddefault}{\updefault}{\color[rgb]{0,0,0}$c$}%
}}}}
\put(6151,-5086){\makebox(0,0)[lb]{\smash{{\SetFigFont{6}{7.2}{\rmdefault}{\mddefault}{\updefault}{\color[rgb]{0,0,0}$b$}%
}}}}
\end{picture}% 

&

\begin{picture}(0,0)%
\includegraphics{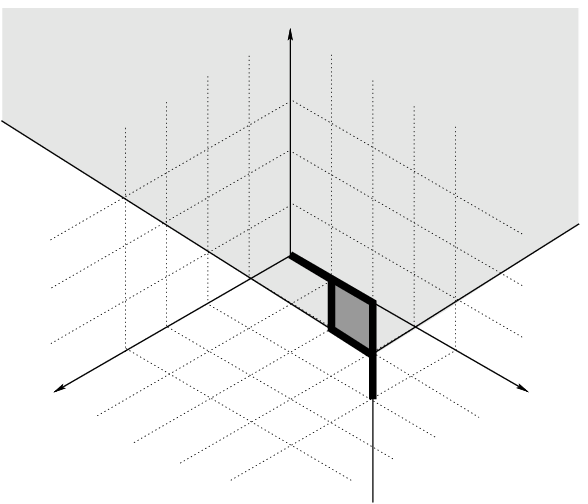}%
\end{picture}%
\setlength{\unitlength}{1302sp}%
\begingroup\makeatletter\ifx\SetFigFont\undefined%
\gdef\SetFigFont#1#2#3#4#5{%
  \reset@font\fontsize{#1}{#2pt}%
  \fontfamily{#3}\fontseries{#4}\fontshape{#5}%
  \selectfont}%
\fi\endgroup%
\begin{picture}(8444,7244)(1779,-7583)
\put(2401,-5611){\makebox(0,0)[lb]{\smash{{\SetFigFont{6}{7.2}{\rmdefault}{\mddefault}{\updefault}{\color[rgb]{0,0,0}$x_1$}%
}}}}
\put(9226,-5686){\makebox(0,0)[lb]{\smash{{\SetFigFont{6}{7.2}{\rmdefault}{\mddefault}{\updefault}{\color[rgb]{0,0,0}$x_2$}%
}}}}
\put(6751,-4936){\makebox(0,0)[lb]{\smash{{\SetFigFont{6}{7.2}{\rmdefault}{\mddefault}{\updefault}{\color[rgb]{0,0,0}$V$}%
}}}}
\put(5401,-1036){\makebox(0,0)[lb]{\smash{{\SetFigFont{6}{7.2}{\rmdefault}{\mddefault}{\updefault}{\color[rgb]{0,0,0}$x_3$}%
}}}}
\put(6151,-1561){\makebox(0,0)[lb]{\smash{{\SetFigFont{6}{7.2}{\rmdefault}{\mddefault}{\updefault}{\color[rgb]{0,0,0}$x_2\oplus 3x_3= 2x_1\oplus 3x_3$}%
}}}}
\put(6076,-3886){\makebox(0,0)[lb]{\smash{{\SetFigFont{6}{7.2}{\rmdefault}{\mddefault}{\updefault}{\color[rgb]{0,0,0}$a$}%
}}}}
\put(7426,-6136){\makebox(0,0)[lb]{\smash{{\SetFigFont{6}{7.2}{\rmdefault}{\mddefault}{\updefault}{\color[rgb]{0,0,0}$c$}%
}}}}
\put(6151,-5161){\makebox(0,0)[lb]{\smash{{\SetFigFont{6}{7.2}{\rmdefault}{\mddefault}{\updefault}{\color[rgb]{0,0,0}$b$}%
}}}}
\end{picture}%

\end{tabular} 
\end{center}
\caption{Valid linear inequality and equality for $V$.}
\label{figurevalid} 
\end{figure}

\begin{example}\label{ejemplofinal}
Consider the semimodule $V\subset \rmax^3$ generated by the vectors 
$a=(0,0,0)^T$, $b=(0,1,-1)^T$ and $c=(0,2,-2)^T$.  
This semimodule is represented by the bounded dark gray region of 
Figure~\ref{figurevalid} together with the segments joining the points 
$a$ and $c$ to it. Then $V^\circ$ is the solution set of a system  
of homogeneous max-plus linear inequalities, and hence also of 
a system of homogeneous max-plus linear equations. More precisely, 
\begin{align*}
V^\circ & = 
\set{(f,g)\in (\rmax^3)^2}{\oplus_i f_i x_i \leq \oplus_j g_j x_j \; , 
\; x=a,b,c} \\
 &=\set{(f,g)\in (\rmax^3)^2}{\oplus_i (f_i\oplus g_i) x_i = 
\oplus_j g_j x_j\; , \; x=a,b,c} \; . 
\end{align*}
We know that it is a finitely generated subsemimodule of $\rmax^6$ 
(see~\cite{butkovicH,gaubert92a,maxplus97}). Solving this system by the 
elimination method (see~\cite{butkovicH} and~\cite{AGG08} for recent 
improvements), we obtain that the polar cone $V^\circ $ is the 
subsemimodule of $\rmax^6$ generated by the columns of the following matrix 
\[
\left(
\begin{array}{cccccccccccccc}
\zero & \zero & \zero &   1   & \zero &   2   &   0   &   0   &   0   & \zero & \zero & \zero & \zero & \zero \cr
  0   & \zero & \zero & \zero &   0   & \zero & \zero & \zero & \zero &   0   & \zero & \zero & \zero & \zero \cr
\zero &   0   &   0   & \zero & \zero & \zero & \zero & \zero & \zero & \zero &   0   & \zero & \zero & \zero \cr
  2   & \zero &   0   & \zero & \zero & \zero & \zero & \zero &   0   & \zero & \zero &   0   & \zero & \zero \cr
\zero &   0   & \zero &   0   & \zero &   0   &   0   & \zero & \zero &   0   & \zero & \zero &   0   & \zero \cr
\zero & \zero & \zero &   1   &   4   &   3   & \zero &   2   & \zero & \zero &   0   & \zero & \zero &   0
\end{array} 
\right) \; .
\]
The ordering of the variables in the system of equations above has been chosen 
such that, if we denote by $f$ the vector whose entries are the first three 
entries of a solution and by $g$ the vector whose entries are the remaining 
three entries, then $(f,g)\in V^\circ$. For instance, 
the fourth column of the matrix above corresponds to the pair of vectors 
$f=(1, \zero , \zero)^T$ and $g=(\zero , 0 ,1)^T$. 
With each pair of vectors $(f,g)\in V^\circ$ is 
associated the linear inequality $\oplus_i f_i x_i \leq \oplus_j g_j x_j$ 
which is said to be {\em valid} for $V$ because it is satisfied by all its 
elements. For the previous pair of vectors, we have the valid inequality 
$1x_1\leq x_2 \oplus 1x_3$, which is represented by the unbounded 
light gray region on the left-hand side of Figure~\ref{figurevalid}. 
Analogously, the congruence $V^\bot$ is the subsemimodule of 
$\rmax^6$ generated by the columns of the following matrix
\[
\left(
\begin{array}{ccccccccccccccccccc}
  0   &   1   & \zero &   0   &   2   & \zero & \zero & \zero & \zero & \zero & \zero & \zero &   2   &   2   &   0   &   0   &   0   & \zero & \zero \cr
  0   &   0   &   0   & \zero & \zero &   0   &   0   &   0   &   0   & \zero &   0   & \zero &   0   & \zero & \zero & \zero & \zero &   0   & \zero \cr
\zero & \zero &   0   &   0   & \zero &   4   &   3   & \zero & \zero &   2   &   1   &   4   & \zero &   3   & \zero &   2   & \zero & \zero &   0   \cr
\zero & \zero & \zero &   0   &   2   & \zero &   2   &   0   & \zero &   0   &   1   & \zero &   2   & \zero &   0   & \zero &   0   & \zero & \zero \cr
  0   &   0   &   0   & \zero &   0   & \zero & \zero &   0   &   0   & \zero &   0   &   0   & \zero &   0   & \zero & \zero & \zero &   0   & \zero \cr
\zero &   1   & \zero & \zero & \zero &   4   &   3   & \zero &   0   &   2   & \zero &   4   & \zero &   3   &   0   &   2   & \zero & \zero &   0
\end{array} 
\right) \; .
\]
In this case, to solve the system of equations for $V^\bot$, 
we have taken the same ordering of the variables as for $V^\circ $. 
Then, for example, the first column corresponds to the equality 
$x_1\oplus x_2 =x_2 $ which  is equivalent to the valid inequality 
$x_1\leq x_2$ (the seventh column of the matrix containing the 
generators of $V^\circ $). 

Note that, as mentioned in the introduction, 
any valid inequality $\oplus_i f_i x_i \leq \oplus_j g_j x_j$ 
for $V$ implies a valid equality for $V$ 
(meaning that it is satisfied by all its elements), 
namely $\oplus_i (f_i \oplus g_i ) x_i= \oplus_j g_j x_j$. 
However, not all the valid equalities for $V$ can be obtained in this way. 
For instance, the equality $x_2\oplus 3x_3=2x_1\oplus3 x_3$ is 
satisfied by all the elements of $V$.  It corresponds to the seventh column 
of the latter matrix and it is represented by the unbounded light gray 
region on the right-hand side of Figure~\ref{figurevalid}. 
This equality cannot be obtained as a max-plus linear combination of the 
equalities that can be derived from the inequalities in $V^\circ $. 
\end{example}

\paragraph{Acknowledgment}
The authors would like to thank the anonymous referees for 
their valuable comments and suggestions, 
which improved this manuscript.

\newcommand{\etalchar}[1]{$^{#1}$}

\end{document}